\documentclass{article}
\usepackage{amsmath,amssymb}
\newtheorem{theorem}{Theorem}[section]

\newtheorem{lemma}[theorem]{Lemma}
\newtheorem{definition}[theorem]{Definition}
\newtheorem{corollary}[theorem]{Corollary}
\newtheorem{proposition}[theorem]{Proposition}
\newcommand{\qed}{\rule{1mm}{3mm}}     
\newcommand{\N}{{\mathbb N}}
\newcommand{\R}{{\mathbb R}}
\newcommand{\oder}{\;\vee\;}

\newenvironment{proof}{\vspace*{\parsep}\noindent {\bf proof:}}{\qed\\[1em]}
\newcommand{\CZF}{CZF}
\newcommand{\voll}[2]{{\mathbf{mv}}(\phantom{}^{#1}{#2})}
\newcommand{\hoch}[2]{\phantom{}^{#1}{#2}}
\newcommand{\Fun}{{\mathbf{Fun}}}
\newcommand{\dom}{{\mathbf{dom}}}
\newcommand{\ran}{{\mathbf{ran}}}
 \newcommand{\ini}{\in}
 \newcommand{\inn}{\in}
 \newcommand{\paar}[1]{\langle #1\rangle}
 \newcommand{\ttt}{\em }

\begin{document}

\title{On the Constructive Dedekind Reals}
\author{Robert S. Lubarsky\\
{\small Department of Mathematical Sciences, Florida Atlantic University}\\
{\small 777 Glades Road Boca Raton, FL 33431, USA} \\
{\tt \scriptsize rlubarsk@fau.edu}\\
  Michael
Rathjen\footnote{This material is based upon work supported by the
National
Science Foundation under Award No. DMS-0301162.}\\
{\small Department of Pure Mathematics, University of Leeds}\\
{\small Leeds LS2 9JT, United Kingdom} \\
{\tt \scriptsize rathjen@math.ohio-state.edu}}

\maketitle

\begin{abstract} In order to build the collection of Cauchy reals
as a set in constructive set theory, the only Power Set-like
principle needed is Exponentiation. In contrast, the proof that
the Dedekind reals form a set has seemed to require more than
that. The main purpose here is to show that Exponentiation alone
does not suffice for the latter, by furnishing a Kripke model of
constructive set theory, CZF with Subset Collection replaced by
Exponentiation, in which the Cauchy reals form a set while the
Dedekind reals constitute a proper class.
\end{abstract}

\section{Introduction}
In classical mathematics, one principal approach to defining the
real numbers is to use equivalence classes of Cauchy sequences of
rational numbers, and the other is the method of Dedekind cuts
wherein reals appear as subsets of $\mathbb Q$ with special
properties. Classically the two methods are equivalent in that the
resulting field structures are easily shown to be isomorphic. As
often happens in an intuitionistic setting, classically equivalent
notions fork. Dedekind reals give rise to several demonstrably
different collections of reals when only intuitionistic logic is
assumed (cf. \cite{troelstra}, Ch.5, Sect.5). Here we shall be
concerned with the most common and fruitful notion of Dedekind
real which crucially involves the (classically superfluous)
condition of locatedness of cuts. These Dedekind reals are
sometimes referred to as the {\em constructive Dedekind reals} but
we shall simply address them as the {\em Dedekind reals}. Even in
 intuitionistic set theory,  with a
little bit of help from the countable axiom of choice
(${\mathbf{AC}}({\mathbb N},2)$\footnote{$\forall\,
r\subseteq{\mathbb N}\times2[\forall n\in{\mathbb N}\,\exists i\in
\{0,1\}\;\langle n,i\rangle\in r\;\rightarrow\;\exists f:{\mathbb
N}\to 2\;\forall n\in{\mathbb N}\;\langle n,f(n)\rangle\in r]$.}
suffices; see \cite{ac4}, 8.25),
 ${\mathbb R}^d$ and ${\mathbb R}^c$ are isomorphic
 (where ${\mathbb R}^d$ and ${\mathbb R}^c$ denote the
  collections of Dedekind reals and Cauchy reals,
  respectively). As ${\mathbb R}^c$ is canonically
  embedded in ${\mathbb R}^d$ we can view ${\mathbb R}^c$ as a
  subset of ${\mathbb R}^d$ so that the latter result can be stated
  as ${\mathbb R}^d={\mathbb R}^c$.
The countable axiom of choice is accepted in Bishop-style
  constructive mathematics but cannot be assumed in all intuitionistic contexts.
   Some  choice is necessary for
  equating ${\mathbb R}^d$ and ${\mathbb R}^c$
   as there are sheaf models of
  higher order intuitionistic logic  in which ${\mathbb R}^d$ is not isomorphic
  to ${\mathbb
  R}^c$ (cf. \cite{FH}). This paper will show that the
  difference between ${\mathbb R}^d$ and ${\mathbb R}^c$ can be of
  a grander scale. When is the continuum a set? The standard,
  classical construction of $\mathbb R$ as a set uses Power Set.
  Constructively, the weaker principle of Subset Collection (in the context
  of the axioms of Constructive
  Zermelo-Fraenkel Set Theory CZF) suffices, as does even the
  apparently even weaker principle of Binary Refinement
  \cite{CIS}. In contrast, we shall demonstrate that there is a Kripke
  model of CZF with
  Exponentiation in lieu of Subset Collection in which the Cauchy
  reals form a set while the Dedekind reals constitute a proper
  class. This shows that Exponentiation and  Subset Collection Axiom have markedly
  different consequences for the theory of Dedekind reals.

  This paper proves the following theorems:
  \begin{theorem} (Fourman-Hyland \cite{FH})
IZF$_{Ref}$ does not prove that the Dedekind reals equal the
Cauchy reals.
\end{theorem}

\begin{theorem}
CZF$_{Exp}$ (i.e. CZF with Subset Collection replaced by
Exponentiation) does not prove that the Dedekind reals are a set.
\end{theorem}

Even though the proof of the first theorem given here could be
converted easily to the original Fourman-Hyland proof of the same,
it is still included because the conversion in the other
direction, from the original sheaf proof to the current Kripke
model, is not obvious (to us at least); one might well want to
know what the Kripke model proof of this theorem is. Furthermore,
it is helpful as background to understand the construction of the
second proof. While the second proof could similarly be turned
into a purely topological argument, albeit of a non-standard type,
unlike Gauss, we do not wish to cover our tracks. The original
intuition here was the Kripke model -- indeed, we know of no other
way to motivate the unusual topological semantics and term
structure -- and so it might be of practical utility to have that
motivation present and up front. These benefits of presenting the
Kripke constructions notwithstanding, this article is
reader-friendly enough so that anyone who wanted to could simply
skip the sections on constructing the models and go straight to
the definitions of topological semantics (mod exchanging later on
a few ``true at node $r$"s with ``forced by some neighborhood of
$r$"s).

The paper is organized as follows. After a brief review of
  Constructive Zermelo-Fraenkel Set Theory and  notions of real
  numbers, section 2 features a Kripke model of $IZF_{Ref}$
  in which ${\mathbb R}^d\ne {\mathbb R}^c$. Here  $IZF_{Ref}$
  denotes
  Intuitionistic Zermelo-Fraenkel Set Theory with the Reflection
  schema.\footnote{Reflection, Collection and Replacement are
  equivalent in classical set theory. Intuitionistically,
  Reflection implies Collection which in turn implies Replacement, however,
these implications cannot be reversed (see \cite{FS} for the
latter).} In section 3 the model of section 2 undergoes
refinements and
  pivotally techniques of \cite{RSL1} are put to use to engender a model of
  $CZF_{Exp}$ in which ${\mathbb R}^d$ is a proper class.

  \subsection{Constructive Zermelo-Fraenkel Set Theory}
In this subsection we will summarize the language and axioms for
$\CZF$.  The language of $\CZF$ is the same first order language
as that of classical Zermelo-Fraenkel Set Theory, $ZF$ whose only
non-logical symbol is $\in$. The logic of $\CZF$ is intuitionistic
first order logic with equality. Among its non-logical axioms are
{\it Extensionality}, {\it Pairing} and {\it Union} in their usual
forms.
  $\CZF$ has
additionally axiom schemata which we will now proceed to
summarize.
\paragraph{Infinity:}
$\exists x\forall u\bigl[u\ini x\leftrightarrow
\bigl(\emptyset=u\oder \exists v\ini x\;u=v+1 \bigr)\bigr]$ where
$v+1=v\cup\{v\}$.
\paragraph{Set Induction:}
$\forall x[\forall y \in x \phi (y) \rightarrow \phi (x)]
\rightarrow \forall x \phi (x)$
\paragraph{Bounded Separation:}
$\phantom{AAAA}\forall a \exists b \forall x
[x \in b \leftrightarrow x \in a \wedge \phi (x)]$ \\[1ex]
for all {\it bounded} formulae  $\phi$. A set-theoretic formula is
{\em bounded} or {\em restricted} if it is constructed from prime
formulae using $\neg,\wedge,\vee,\rightarrow,\forall x\ini y$ and
$\exists x\ini y$ only.
\paragraph{Strong Collection:}  For all formulae $\phi$,
$$\forall a
\bigl[\forall x \in a \exists y \phi (x,y)\;\rightarrow\; \exists
b \,[\forall x \in a\, \exists y \in b\, \phi (x,y) \wedge \forall
y \in b\, \exists x \in a \,\phi (x,y)]\bigr].$$
\paragraph{Subset Collection:}  For all formulae $\psi$,
\begin{eqnarray*}\lefteqn{\forall a \forall b \exists c
\forall u \,\bigl[\forall x \in a \,\exists y \in b\;\psi (x,y,u)
\,\rightarrow \,}\\ && \exists d \in c \,[\forall x \in a\,
\exists y \in d\, \psi (x,y,u) \wedge \forall y \in d\, \exists x
\in a \,\psi (x,y,u)]\bigr].\end{eqnarray*}
  The Subset Collection
schema easily qualifies as the most intricate axiom of $\CZF$. To
explain this axiom in different terms, we introduce the notion of
{\em fullness} (cf. \cite{ac1}).
\begin{definition}{\em As per usual, we use $\langle x,y\rangle$ to denote
the ordered pair of $x$ and $y$. We use $\Fun(g)$, $\dom(R)$,
$\ran(R)$ to convey that $g$ is a function and to denote the
domain and range of any relation $R$, respectively.

For sets $A,B$ let $A \times B$ be the cartesian product of $A$
and $B$, that is the set of ordered pairs $\langle x,y\rangle$
with $x\inn A$ and $y\inn B$. Let $\hoch{A}{B}$ be the class of
all functions with domain $A$ and with range contained in $B$. Let
$\voll{A}{B}$ be the class of all sets $R\subseteq A\times B$
satisfying $\forall u\ini A\,\exists v\ini B\,\paar{u,v}\in R$. A
set $C$ is said to be {\em full in} $\voll{A}{B}$ if $C\subseteq
\voll{A}{B}$ and $$\forall R\in\voll{A}{B}\,\exists S\ini C\,
S\subseteq R.$$

The expression $\voll{A}{B}$ should be read as the collection of
{\em multi-valued functions} from the set $A$ to the set $B$.

Additional axioms we shall consider are:
\paragraph{Exponentiation:} $\forall x\forall y\exists z\, z=\hoch{x}{y}$.
\paragraph{Fullness:}
$\forall x\forall y\exists z\,\;\mbox{\em $z$ is full in
$\voll{x}{y}$}$. }\end{definition} The next result provides an
equivalent rendering of Subset Collection.

\begin{proposition}\label{ober2.3}
 Let $\CZF^-$ be $\CZF$ without {\em Subset Collection}.
\begin{itemize}
\item[(i)] $\CZF^- + \mbox{\ttt Subset
Collection}\; = \; \CZF^- + \mbox{\ttt Fullness}.$ \item[(ii)]
$\CZF\vdash\,\mbox{\ttt
Exponentiation}$.\end{itemize}\end{proposition}

\proof \cite{ac1}, Proposition 2.2. (The equality in $(i)$ is as
theories: that is, they both prove the same theorems.)\qed
  \subsection{The Cauchy and Dedekind reals}

  \begin{definition} A {\em fundamental sequence} is a sequence
  $(r_n)_{n\in{\mathbb N}}$ of rationals, together with is a {\em
  (Cauchy-)modulus} $f:\N\to\N$ such that
    $$\forall k\,\forall
    m, n\geq f(k)\;\; |r_m-r_n|<\frac1{2^k},$$
    where all quantifiers range over $\N$.

    Two fundamental sequences $(r_n)_{n\in{\mathbb N}},(s_n)_{n\in{\mathbb N}}$
    are said to coincide (in symbols $\approx$)
    if
    $$\forall k\exists n\forall m\geq
    n\;\;|r_m-s_m|<\frac1{2^k}.$$
    $\approx$ is indeed an equivalence relation on
    fundamental sequences. The set of Cauchy reals $\R^c$ consists
    of the equivalence classes of fundamental sequences relative
    to
    $\approx$. For the equivalence class of $(r_n)_{n\in{\mathbb
    N}}$ we use the notation $(r_n)_{n\in{\mathbb N}}/\approx$.
    \end{definition}

The development of  the theory of Cauchy reals in
    \cite{troelstra}, Ch.5, Sect.2-4 can be carried out on the
    basis of $CZF_{Exp}$. Note that the axiom AC-NN! \footnote
    {$(\forall m \in \mathbb{N} \; \exists ! n \in \mathbb{N}
    \; \phi(m,n))
    \rightarrow (\exists f:\mathbb{N} \rightarrow \mathbb{N}
    \; \forall m \in \mathbb{N} \; \phi(m, f(m)))$} is deducible
    in $CZF_{Exp}$.

    \begin{definition} Let $S\subseteq{\mathbb Q}$. $S$ is called
    a {\em left cut} or {\em Dedekind real} if the following
    conditions are satisfied:
    \begin{enumerate}
    \item $\exists r(r\in S)\;\wedge\;\exists r'(r'\notin S)$
    {\em (boundedness)}
    \item $\forall r\in S\,\exists r'\in S\,(r<r')$ {\em
    (openness)}
    \item $\forall rs\in{\mathbb Q}\,[r<s\rightarrow r\in S\vee
    s\notin S]$ {\em (locatedness)}
\end{enumerate}
\end{definition} For $X\subseteq {\mathbb Q}$ define
$X^<:=\{s\in{\mathbb Q}: \exists r\in X \;s<r\}$.  If $S$ is a
left cut it follows from openness and locatedness that $S=S^<$.

\begin{lemma} Let $r=(r_n)_{n\in{\mathbb N}}$ and  $r'=(r'_n)_{n\in{\mathbb N}}$
be fundamental sequences of rationals. Define $$X_r:=\{s\in
{\mathbb Q}: \exists m \;s<(r_{f(m)}-\frac1{2^m})\}.$$ We then
have
\begin{enumerate} \item $X_r$ is a Dedekind real.
\item $X_r=X_{r'}$ if and only if $(r_n)_{n\in{\mathbb
N}}\approx(r'_n)_{n\in{\mathbb N}}$. \item $\R^c$ is a subfield of
$\R^d$ via the mapping $(r_n)_{n\in{\mathbb
N}}/\approx\;\;\mapsto\; X_r$.
\end{enumerate}
\end{lemma}
{\bf Proof:} Exercise or see \cite{ac4}, Section 8.4. \qed

\section{${\mathbb R}^d\ne {\mathbb R}^c$}

\begin{theorem} (Fourman-Hyland \cite{FH})
IZF$_{Ref}$ does not prove that the Dedekind reals equal the
Cauchy reals.
\end{theorem}

\subsection{Construction of the Model}
Let $M_{0} \prec M_{1} \prec$ ... be an $\omega$-sequence of
models of ZF set theory and of elementary embeddings among them,
as indicated, such that the sequence from $M_{n}$ on is definable
in $M_{n}$, and such that each thinks that the next has
non-standard integers. Notice that this is easy to define (mod
getting a model of ZF in the first place): an iterated ultrapower
using any non-principal ultrafilter on $\omega$ will do. (If
you're concerned that this needs AC too, work in L of your
starting model.) We will ambiguously use the symbol $f$ to stand
for any of the elementary embeddings inherent in the
$M_{n}$-sequence.

\begin{definition}
The frame (underlying partial order) of the Kripke model $M$ will
be a (non-rooted) tree with $\omega$-many levels. The nodes on
level $n$ will be the reals from $M_n$. $r'$ is an immediate
successor of $r$ iff $r$ is a real from some $M_n$, $r'$ is a real
from $M_{n+1}$, and $r$ and $r'$ are infinitesimally close; that
is, $f(r) - r'$, calculated in $M_{n+1}$ of course, is
infinitesimal, calculated in $M_n$ of course. In other words, in
$M_n$, $r$ is that standard part of $r'$.
\end{definition}

The Kripke structure will be defined like a forcing extension in
classical set theory. That is, there will be a ground model, terms
that live in the ground model, and an interpretation of those
terms, which, after modding out by =, is the final model $M$.
Since the current construction is mostly just a re-phrasing of the
topological, i.e. Heyting-valued, model of \cite{FH}, the
similarity with forcing, i.e. Boolean-valued models, is not merely
an analogy, but essentially the same material, and so it makes
some sense to present it the way people are used to it.

\begin{definition}
The ground Kripke model has, at each node of level $n$, a copy of
$M_{n}$. The transition functions (from a node to a following
node) are the elementary embeddings given with the original
sequence of models (and therefore will be notated by $f$ again).
\end{definition}

Note that by the elementarity of the extensions, this Kripke model
is a model of classical ZF. More importantly, the model restricted
to any node of level $n$ is definable in $M_{n}$, because the
original $M$-sequence was so definable.

\begin{definition}
The terms are defined at each node separately. For a node at level
$n$, the terms are defined in $M_n$, inductively on the ordinals
in $M_n$. At any stage $\alpha$, a term of stage $\alpha$ is a set
$\sigma$ of the form $\{ \langle \sigma_{i}, J_{i} \rangle \mid i
\in I \}$, where I is some index set, each $\sigma_{i}$ is a term
of stage $< \alpha$, and each J$_{i}$ is an open subset of the
real line.
\end{definition}
(Often the terms at stage $\alpha$ are defined to be functions
from the terms of all stages less than $\alpha$, as opposed to the
relations above, which may be non-total and multi-valued. This
distinction makes absolutely no difference. Such a relation can be
made total by sending all terms not yet in the domain to the empty
set, and functional by taking unions of second components.)

Intuitively, each open set $J$ is saying ``the generic real is in
me." Also, each node $r$ is saying ``I am the generic, or at least
somebody in my infinitesimal universe is." So at node $r$, $J$
should count at true iff $r \in J$. These intuitions will appear
later as theorems. (Well, lemmas.)

The ground model can be embedded in this term structure: for $x
\in M_n$, its canonical name $\hat{x}$ is defined inductively as
$\{ \langle \hat{y}, \mathbb{R} \rangle \mid y \in x \}$. Terms of
the form $\hat{x}$ are called ground model terms.

Notice that the definition of the terms given above is uniform
among the $M_n$'s, and so any term at a node gets sent by the
transition function $f$ to a corresponding term at any given later
node. Hence we can use the same functions $f$ yet again as the
transition functions for this term model. (Their coherence on the
terms follows directly from their coherence on the original
$M_n$'s.)

At this point in the construction of the Kripke model, we have the
frame, a universe (set of objects) at each node, and the
transition functions. Now we need to define the primitive
relations at each node. In the language of set theory, these are
=$_{M}$ and $\in_{M}$ (the subscript being used to prevent
confusion with equality and membership of the ambient models
$M_{n}$). This will be done via a forcing relation $\Vdash$.

\begin{definition}
$J \Vdash \sigma =_{M} \tau$ and $J \Vdash \sigma \in_{M} \tau$
are defined inductively on $\sigma$ and $\tau$, simultaneously for
all open sets of reals $J$:

$J \Vdash \sigma =_{M} \tau$ iff for all $\langle \sigma_{i},
J_{i} \rangle \in \sigma \; J \cap J_{i} \Vdash \sigma_{i} \in_{M}
\tau$ and vice versa

$J \Vdash \sigma \in_{M} \tau$ iff for all $r \in J$ there is a
$\langle \tau_{i}, J_{i} \rangle \in \tau$ and $J'$ such that $r
\in J' \cap J_{i} \Vdash \sigma =_{M} \tau_{i}$
\end{definition}

(We will later extend this forcing relation to all formulas.)

Note that these definitions, for $J, \sigma, \tau \in M_n$, can be
evaluated in $M_n$, without reference to $M_{n+1}$ or to future
nodes or anything. Therefore $J \Vdash \phi$ (according to $M_n$)
iff $f(J) \Vdash f(\phi)$ (according to $M_{n+1}$), by the
elementarity of $f$. So we can afford to be vague about where
various assertions are evaluated, since by this elementarity it
doesn't matter. (The same will be true when we extend forcing to
all formulas.)

\begin{definition}
At node $r$, for any two terms $\sigma$ and $\tau$, $\sigma =_{M}
\tau$ iff, for some $J$ with $r \in J$, $J \Vdash \sigma =_{M}
\tau$.

Also, at $r$, $\sigma \in_{M} \tau$ iff for some $J$ with $r \in
J$, $J \Vdash \sigma \in_{M} \tau$.
\end{definition}
\underline{\bf Notation} Satisfaction (in the sense of Kripke
semantics) at node $r$ will be notated with $\models$, as in $``r
\models \sigma = \tau"$. This should not be confused with the
forcing relation $\Vdash$, even though the latter symbol is often
used in the literature for Kripke satisfaction.

Thus we have a first-order structure at each node.

To have a Kripke model, the transition functions $f$ must also
respect this first-order structure, =$_{M}$ and $\in_{M}$; to wit:

\begin{lemma}
f is an $=_{M}$ and $\in_{M}$-homomorphism. That is, if $\sigma
=_{M} \tau$ then $f(\sigma) =_{M} f(\tau$), and similarly for
$\in_{M}$.
\end{lemma}
\begin{proof}
If $\sigma =_{M} \tau$ then let $J$ be such that $r \in J$ and $J
\Vdash \sigma =_{M} \tau$. Then $f(J)$ is open, $r' \in f(J)$
because $r'$ is infinitesimally close to $r$, and $f(J) \Vdash
f(\sigma) =_{M} f(\tau$) by elementarity. Hence $f(\sigma) =_{M}
f(\tau)$. Similarly for $\in_{M}$.
\end{proof}

We can now conclude that we have a Kripke model.

\begin{lemma} \label{equalitylemma}
This Kripke model satisfies the equality axioms:
\begin{enumerate}
\item $\forall x \; x=x$

\item $\forall x, y \; x=y \rightarrow y=x$

\item $\forall x, y, z \; x=y \wedge y=z \rightarrow x=z$

\item $\forall x, y, z \; x=y \wedge x \in z \rightarrow y \in z$

\item $\forall x, y, z \; x=y \wedge z \in x \rightarrow z \in y.$

\end{enumerate}

\end{lemma}
\begin{proof}
1: It is easy to show with a simultaneous induction that, for all
$J$ and $\sigma, J \Vdash \sigma =_{M} \sigma$, and, for all
$\langle \sigma_{i}, J_{i} \rangle \in \sigma, J \cap J_{i} \Vdash
\sigma_{i} \in_{M} \sigma$.

2: Trivial because the definition of $J \Vdash \sigma =_{M} \tau$
is itself symmetric.

3: For this and the subsequent parts, we need some lemmas.
\begin{lemma} If $J' \subseteq J \Vdash \sigma =_{M}
\tau$ then $J' \Vdash \sigma =_{M} \tau$, and similarly for
$\in_{M}$.
\end{lemma}
\begin{proof} By induction on $\sigma$ and $\tau$.
\end{proof}

\begin{lemma} If $J \Vdash \rho =_{M} \sigma$ and $J \Vdash
\sigma =_{M} \tau$ then $J \Vdash \rho =_{M} \tau$.
\end{lemma}
\begin{proof} Again, by induction on terms. Let
$\langle \rho_{i}, J_{i} \rangle \in \rho$. Then $J \cap J_{i}
\Vdash \rho_{i} \in_{M} \sigma$, i.e. for all $r \in J \cap J_{i}$
there are $\langle \sigma_{j}, J_{j} \rangle \in \sigma$ and $J'
\subseteq J \cap J_{i}$ such that $r \in J' \cap J_{j} \Vdash
\rho_{i} =_{M} \sigma_{j}$. Fix any $r \in J \cap J_{i}$, and let
$\langle \sigma_{j}, J_{j} \rangle \in \sigma$ and $J'$ be as
given. By hypothesis, $J \cap J_{j} \Vdash \sigma_{j} \in_{M}
\tau$. So let $\langle \tau_{k}, J_{k} \rangle \in \tau$ and
$\hat{J} \subseteq J \cap J_{j}$ be such that $r \in \hat{J} \cap
J_{k} \Vdash \sigma_{j} =_{M} \tau_{k}$. Let $\tilde{J}$ be $J'
\cap \hat{J} \cap J_{j}$. Note that $\tilde{J} \subseteq J \cap
J_{i}$, and that $r \in \tilde{J} \cap J_{k}$. It remains only to
show that $\tilde{J} \cap J_{k} \Vdash \rho_{i} =_{M} \tau_{k}$.
Observing that $\tilde{J} \cap J_{k} \subseteq J' \cap J_{j},
\hat{J} \cap J_{k}$, it follows by the previous lemma that
$\tilde{J} \cap J_{k} \Vdash \rho_{i} =_{M} \sigma_{j}, \sigma_{j}
=_{M} \tau_{k}$, from which the desired conclusion follows by the
induction.
\end{proof}

Returning to proving property 3, the hypothesis is that for some
$J$ and $K$ containing $r$, $J \Vdash \rho =_{M} \sigma$ and $K
\Vdash \sigma =_{M} \tau$. By the first lemma, $J \cap K \Vdash
\rho =_{M} \sigma, \sigma =_{M} \tau$, and so by the second, $J
\cap K \Vdash \rho =_{M} \tau$, which suffices.

4: Let $J \Vdash \rho =_{M} \sigma$ and $K \Vdash \rho \in_{M}
\tau$. We will show that $J \cap K \Vdash \sigma \in_{M} \tau$.
Let $r \in J \cap K$. By hypothesis, let $\langle \tau_{i}, J_{i}
\rangle \in \tau, J' \subseteq K$ be such that $r \in J' \cap
J_{i} \Vdash \rho =_{M} \tau_{i}$; without loss of generality $J'
\subseteq J$. By the first lemma, $J' \cap J_{i} \Vdash \rho =_{M}
\sigma$, and by the second, $J' \cap J_{i} \Vdash \sigma =_{M}
\tau_{i}$.

5: Similar, and left to the reader.
\end{proof}

With this lemma in hand, we can now mod out by =$_{M}$, so that the
symbol ``=" is interpreted as actual set-theoretic equality. We
will henceforth drop the subscript $_{M}$ from = and $\in$,
although we will not distinguish notationally between a term
$\sigma$ and the model element it represents, $\sigma$'s
equivalence class.

Note that, at any node $r$ of level $n$, the whole structure $M$
restricted to $r$ and its successors is definable in $M_n$,
satisfaction relation $\models$ and all. This will be useful when
showing below that IZF holds. For instance, to show Separation,
satisfaction $r \models \phi(x)$ will have to be evaluated in
order to define the right separation term in $M_n$, and so
satisfaction must be definable in $M_n$.

\subsection{The Forcing Relation}
The primitive relations = and $\in$ were defined in terms of open
sets $J$. To put it somewhat informally, at $r$, $\sigma =$ or
$\in \tau$ if this is forced by a true set, and a set $J$ is true
at $r$ if $r \in J$. In fact, this phenomenon propagates to
non-primitive formulas. To show this, we extend the forcing
relation $J \Vdash \phi$ from primitive to all (first-order,
finitary) formulas. Then we prove as a lemma, the Truth Lemma,
what was taken as a definition for the primitive formulas, that $r
\models \phi$ iff $J \Vdash \phi$ for some $J$ containing $r$.

\begin{definition}
$J \Vdash \phi$ is defined inductively on $\phi$:

$J \Vdash \sigma = \tau$ iff for all $\langle \sigma_{i}, J_{i}
\rangle \in \sigma \; J \cap J_{i} \Vdash \sigma_{i} \in \tau$ and
vice versa

$J \Vdash \sigma \in \tau$ iff for all $r \in J$ there is a
$\langle \tau_{i}, J_{i} \rangle \in \tau$ and $J'$ such that $r
\in J' \cap J_{i} \Vdash \sigma = \tau_{i}$

$J \Vdash \phi \wedge \psi$ iff $J \Vdash \phi$ and $J \Vdash
\psi$

$J \Vdash \phi \vee \psi$ iff for all $r \in J$ there is a $J'$
containing $r$ such that $J \cap J' \Vdash \phi$ or $J \cap J'
\Vdash \psi$

$J \Vdash \phi \rightarrow \psi$ iff for all $J' \subseteq J$ if
$J' \Vdash \phi$ then $J' \Vdash \psi$

$J \Vdash \exists x \; \phi(x)$ iff for all $r \in J$ there is a
$J'$ containing $r$ and a $\sigma$ such that $J \cap J' \Vdash
\phi(\sigma)$

$J \Vdash \forall x \; \phi(x)$ iff for all $r \in J$ and $\sigma$
there is a $J'$ containing $r$ such that $J \cap J' \Vdash
\phi(\sigma)$

\end{definition}

\begin{lemma} \label{helpful lemma}
\begin{enumerate}
\item For all $\phi \; \; \emptyset \Vdash \phi$.
\item If $J' \subseteq J \Vdash \phi$ then $J' \Vdash \phi$.
\item If $J_{i} \Vdash \phi$ for all $i$ then $\bigcup_{i} J_{i}
\Vdash \phi$.
\item $J \Vdash \phi$ iff for all $r \in J$ there is a $J'$
containing $r$ such that $J \cap J' \Vdash \phi$.
\end{enumerate}
\end{lemma}

\begin{proof}
1. Trivial induction. The one observation to make regards
negation, not mentioned above. As is standard, $\neg \phi$ is
taken as an abbreviation for $\phi \rightarrow \bot$, where $\bot$
is any false formula. Letting $\bot$ be ``0=1", observe that
$\emptyset \Vdash \bot$.

2. Again, a trivial induction.

3. Easy induction. The one case to watch out for is $\rightarrow$,
where you need to invoke the previous part of this lemma.

4. Trivial, using 3.
\end{proof}

\begin{lemma} Truth Lemma: For any node $r$, $r \models \phi$
iff $J \Vdash \phi$ for some $J$ containing $r$.
\end{lemma}

\begin{proof}
Again, by induction on $\phi$, this time in detail for a change.

In all cases, the right-to-left direction (``forced implies true")
is pretty easy, by induction. (Note that only the $\rightarrow$
case needs the left-to-right direction in this induction.) Hence
in the following we show only left-to-right (``if true at a node
then forced").

=: This is exactly the definition of =.

$\in$: This is exactly the definition of $\in$.

$\wedge$: If $r \models \phi \wedge \psi$, then $r \models \phi$
and $r \models \psi$. Inductively let $r \in J \Vdash \phi$ and $r
\in J' \Vdash \psi. \; J \cap J'$ suffices.

$\vee$: If $r \models \phi \vee \psi$, then without loss of
generality $r \models \phi$ . Inductively let $r \in J \Vdash
\phi$. $J$ suffices.

$\rightarrow$: Suppose to the contrary $r \models \phi \rightarrow
\psi$ but no open set containing $r$ forces such. Work in an
infinitesimal neighborhood $J$ around $r$. Since $J \not\Vdash
\phi \rightarrow \psi$ there is a $J' \subseteq J$ such that $J'
\Vdash \phi$ but $J' \not\Vdash \psi$.  By the previous part of
this lemma, there is an $r' \in J'$ such that no open set
containing $r'$ forces $\psi$. At the node $r'$, by induction, $r'
\not\models \psi$, even though $r' \models \phi$ (since $r' \in J'
\Vdash \phi$). This contradicts the assumption on $r$ (i.e. that
$r$ $\models \phi \rightarrow \psi$), since $r'$ extends $r$ (as
nodes).

$\exists$: If $r \models \exists x \; \phi(x)$ then let $\sigma$
be such that $r \models \phi(\sigma)$. Inductively there is a $J$
containing $r$ such that $J \Vdash \phi(\sigma)$. $J$ suffices.

$\forall$: Suppose to the contrary $r \models \forall x \;
\phi(x)$ but no open set containing $r$ forces such. Work in an
infinitesimal neighborhood $J$ around $r$. Since $J \not\Vdash
\forall x \; \phi(x)$ there is an $r' \in J$ and $\sigma$ such
that for all $J'$ containing $r' \; J \cap J' \not\Vdash
\phi(\sigma)$. That is, no open set containing $r'$ forces
$\phi(\sigma)$. Hence at the node $r'$, by induction, $r'
\not\models \phi(\sigma)$. This contradicts the assumption on $r$
(i.e. that $r$ $\models \forall x \; \phi(x)$).

\end{proof}

\subsection{The Final Proof}
We now want to show that our model $M$ satisfies certain global
properties. If it had a bottom element $\bot$, then we could
express what we want by saying $\bot \models \phi$ for certain
$\phi$. But it doesn't. Hence we use the abbreviation $M \models
\phi$ for ``for all nodes $r$, $r \models \phi$."

\begin{theorem}
$M \models$ IZF$_{Ref}$.
\end{theorem}
\begin{proof}
Note that, as a Kripke model, the axioms of intuitionistic logic
are satisfied, by general theorems about Kripke models.

\begin{itemize}
\item Infinity: $\hat \omega$ will do. (Recall
that the canonical name $\hat{x}$ of any set $x \in M_n$ is
defined inductively as $\{ \langle \hat{y}, \mathbb{R} \rangle
\mid y \in x \}.)$
\item Pairing: Given $\sigma$ and $\tau$, $\{ \langle \sigma,
{\mathbb R} \rangle , \langle \tau, {\mathbb R} \rangle \}$ will
do.
\item Union: Given $\sigma$, $\{ \langle \tau, J \cap J_i \rangle
\mid$ for some $\sigma_i, \; \langle \tau, J \rangle \in \sigma_i$
and $\langle \sigma_i, J_i \rangle \in \sigma \}$ will do.

\item Extensionality: We need to show that $\forall x \;
\forall y \; [\forall z \; (z \in x \leftrightarrow z \in y)
\rightarrow x = y]$. So let $\sigma$ and $\tau$ be any terms at a
node $r$ such that $r \models ``\forall z \; (z \in \sigma
\leftrightarrow z \in \tau)"$. We must show that $r \models
``\sigma = \tau"$. By the Truth Lemma, let $r \in J \Vdash
``\forall z \; (z \in \sigma \leftrightarrow z \in \tau)"$; i.e.
for all $r' \in J, \rho$ there is a $J'$ containing $r'$ such that
$J \cap J' \Vdash \rho \in \sigma \leftrightarrow \rho \in \tau$.
We claim that $J \Vdash ``\sigma = \tau"$, which again by the
Truth Lemma suffices. To this end, let $\langle \sigma_i, J_i
\rangle$ be in $\sigma$; we need to show that $J \cap J_i \Vdash
\sigma_i \in \tau$. Let $r'$ be an arbitrary member of $J \cap
J_i$ and $\rho$ be $\sigma_i$. By the choice of $J$, let $J'$
containing $r'$ be such that $J \cap J' \Vdash \sigma_i \in \sigma
\leftrightarrow \sigma_i \in \tau$; in particular, $J \cap J'
\Vdash \sigma_i \in \sigma \rightarrow \sigma_i \in \tau$. It has
already been observed in \ref{equalitylemma}, part 1, that $J \cap
J' \cap J_i \Vdash \sigma_i \in \sigma$, so $J \cap J' \cap J_i
\Vdash \sigma_i \in \tau$. By going through each $r'$ in $J \cap
J_i$ and using \ref{helpful lemma}, part 3, we can conclude that
$J \cap J_i \Vdash \sigma_i \in \tau$, as desired. The other
direction ($``\tau \subseteq \sigma"$) is analogous.

\item Set Induction (Schema): Suppose $r \models ``\forall x \;
((\forall y \in x \; \phi(y)) \rightarrow \phi(x))"$, where $r \in
M_n$ ; by the Truth Lemma, let $J$ containing $r$ force as much.
We must show $r \models ``\forall x \; \phi(x)"$. Suppose not.
Using the definition of satisfaction in Kripke models, there is an
$r' \in M_{n'}$ extending (i.e. infinitesimally close to) $r$
(hence in $J$ in the sense of $M_{n'}$) and a $\sigma$ such that
$r' \not \models f(\phi)(\sigma)$ ($f$ the transition from from
node $r$ to $r'$). By elementarity, there is such an $r'$ in
$M_n$. Let $\sigma$ be such a term of minimal $V$-rank among all
$r'$s $\in J$. Fix such an $r'$. By the Truth Lemma (and the
choice of $J$), $r' \models ``(\forall y \in \sigma \; \phi(y))
\rightarrow \phi(\sigma)"$. We claim that $r' \models ``\forall y
\in \sigma \; \phi(y)"$. If not, then for some $r''$ extending
$r'$ (hence in $J$) and $\tau, r'' \models \tau \in f(\sigma)$ and
$r'' \not \models f(\phi)(\tau)$. Unraveling the interpretation of
$\in$, this choice of $\tau$ can be substituted by a term $\tau$
of lower $V$-rank than $\sigma$. By elementarity, such a $\tau$
would exist in $M_n$, in violation of the choice of $\sigma$,
which proves the claim. Hence $r' \models \phi(\sigma)$, again
violating the choice of $\sigma$. This contradiction shows that $r
\models ``\forall x \; \phi(x)"$.

\item Separation (Schema): Let $\phi(x)$ be a formula and $\sigma$
a term. Then $\{ \langle \sigma_i, J \cap J_i \rangle \mid \langle
\sigma_i, J_i \rangle \in \sigma$ and $J \Vdash \phi(\sigma_i) \}$
will do.

\item Power Set: A term $\bar{\sigma}$ is a normal form subset of
$\sigma$ if for all $\langle \sigma_i, \bar{J_i} \rangle \in
\bar{\sigma}$ there is a $J_i \supseteq \bar{J_i}$ such that
$\langle \sigma_i, {J_i} \rangle \in \sigma$. $\{ \langle
\bar{\sigma}, {\mathbb R} \rangle \mid \bar{\sigma}$ is a normal
form subset of $\sigma \}$ will do.

\item Reflection (Schema): Recall that the statement of Reflection
is that for every formula $\phi(x)$ (with free variable $x$ and
unmentioned parameters) and set $z$ there is a transitive set $Z$
containing $z$ such that $Z$ reflects the truth of $\phi(x)$ in
$V$ for all $x \in Z$. So to this end, let $\phi(x)$ be a formula
and $\sigma$ be a set at a node $r$ of level $n$. Let $k$ be such
that the truth of $\phi(x)$ at node $r$ and beyond is $\Sigma_{k}$
definable in $M_n$. In $M_{n}$, let $X$ be a set containing
$\sigma$, $r$, and $\phi$'s parameters such that $X \prec_{k}
M_{n}$. Let $\tau$ be $\{ \langle \rho, {\mathbb R} \rangle \mid
\rho \in X$ is a term\}. $\tau$ will do.
\end{itemize}
\end{proof}

Just as in the case of regular, classical forcing, there is a
generic element. In the case at hand, this generic can be
identified with the term $\{ \langle \hat r, J \rangle \mid $ $r$
is a rational, $J$ is an open interval from the reals, and $r < J
\}$, where $r < J$ if $r$ is less than each element of $J$. We
will call this term $G$. Note that at node $r$ (of level $n$),
every standard (in the sense of $M_n$) rational less than $r$ gets
into $G$, and no standard real greater than $r$ will ever get into
$G$. Of course, non-standard reals infinitesimally close to $r$
are still up for grabs.

It is important in the following that, if $r \models \sigma \in
\mathbb Q$, then there is a rational $q$ in the sense of $M_n$
($n$ $r$'s level) such that $r \models \sigma = q$. That's because
rationals are (equivalence classes of) pairs of naturals, and the
corresponding fact holds for naturals. And that last statement
holds because $M \models ``\hat {\mathbb N}$ is the set of natural
numbers", and the topological space on which the model is built is
connected. Hence, at $r$, a Cauchy sequence of rationals is just
what you'd think: a sequence with domain $\mathbb N$ in the sense
of $M_n$, range $\mathbb Q$ in the sense of $M_n$, with the right
Cauchy condition on it, which gets extended to a larger domain at
successors of $r$.

\begin{proposition}
$M \models ``G$ is a constructive Dedekind real, i.e. a located
left cut".
\end{proposition}
\begin{proof}
First off, $r \models ``r-1 \in G \; \wedge \; r+1 \not \in G"$.
Secondly, if $r \models `` s < t \in G$", then $\langle t, J
\rangle \in G$, where $t < J$ and $r \in J$. Hence $s < J$, so
$\langle s, J \rangle \in G$, and $r \models s \in G$. Finally,
suppose $r \models ``s, t \in {\mathbb Q} \wedge s < t"$. Either
$s < r$ or $r < t$. Since $s$ and $t$ are both standard (in the
sense of $M_n$, $n$ the level of $r$), either $r \models s \in G$
or $r \models t \not \in G$ respectively.
\end{proof}

In order to complete the theorem, we need only prove the
following:
\begin{proposition}
$M \models$ ``The Dedekind real G is not a Cauchy real."
\end{proposition}
\begin{proof}
Recall that a Cauchy sequence is a function $f : \mathbb N
\rightarrow \mathbb Q$ such that for all $k \in \mathbb N$ there
is an $m_k \in \mathbb N$ such that, for all $i, j > m_k, \; f(i)$
and $f(j)$ are within $2^{-k}$ of each other. Classically such a
function $k \mapsto m_k$, called a modulus of convergence, could
be defined from $f$, but not constructively (see \cite{RSL3}).
Often in a constructive setting a real number is therefore taken
to be a pair of a Cauchy sequence and such a modulus (or an
equivalence class thereof). We will prove the stronger assertion
that $G$ is not even the limit of a Cauchy sequence, even without
a modulus. (A Dedekind real $Y$ is the limit of the Cauchy
sequence $f$ exactly when $r \in Y$ iff $r < f(m_k) - 2^{-k}$ for
some $k$, where $m_k$ is an integer as above.)

Suppose $r \models ``f$ is a Cauchy sequence". By the Truth Lemma,
there is an open set $J$ containing $r$ forcing the same. There
are two cases.

CASE I: There is some open set $J'$ containing $r$ forcing a value
$f(m)$ for each integer $m$ in $M_n$ (where $r \in M_n$). In this
case, $f$ is a ground model function; that is, in $M_n$, hence in
each $M_k$ with $k \geq n$, $g(m)$ can be defined as the unique
$l$ such that $J' \Vdash f(m) = \hat l$, and then $J' \Vdash f =
\hat g.$ Since classical logic holds in $M_n$, either $lim(f)$ is
bounded away from $r$, say by a distance of 2$^{-k}$, or it's not.

If it is, then $r \models G \not = lim(f)$, as follows. Let $J''$
be an interval around $r$ of length less than $2^{-k}$. $J''
\Vdash \hat r - 2^{\hat {-k}} \in G \; \wedge \hat r + 2^{\hat
{-k}} \not \in G$, while $f$ stays more than $2^{-k}$ away from
$r$.

If on the other hand $f$ is not bounded away from r, then the
condition ``$s < f(m_k) - 2^{-k}$ for some $k$" becomes simply
$``s < r"$. So then $f$ would witness that $s \in G$ iff $s < r$.
But this is false: if $r'$ is less than $r$ by an infinitesimal
amount, then $r' \models \hat{r'} < \hat r$ but $r' \not \models
\hat {r'} \in G$, and if $r'$ is greater than $r$ by an
infinitesimal amount, and $s$ is between $r$ and $r'$, then $r'
\models \hat s > \hat r$ but $r' \models \hat s \in G$.

CASE II: Not case I. That means that for any interval $J'$ around
$r$, however small, there is some argument $m$ to $f$ such that
$J'$ does not force any value $f(m)$. By elementarity, in
$M_{n+1}$ pick $J'$ to be some infinitesimally small neighborhood
around $r$, and $m$ such an argument. Pick some value $q$ that
$f(m)$ could have and the maximal (hence non-empty, proper, and
open) subset of $J'$ forcing $f(\hat m)=\hat q$. Pick the maximal
(hence non-empty, proper, and open) subset of $J'$ forcing $f(\hat
m)\not =\hat q$. These two subsets must be disjoint, lest the
intersection force a contradiction. But an open interval cannot be
covered by two disjoint, non-empty open sets. Hence there is an
infinitesimal $s$ in neither of those two subsets. Now consider
the Kripke model at node $s$. $f(m)$ is undefined at $s$.
Otherwise, by the Truth Lemma, there would be some interval $J$
containing $s$ such that $J \Vdash f(\hat m) = \hat p$ for some
particular rational $p$. Whether or not $p = q$ would force $s$
into one of the subsets or the other. Therefore, the node $r$
cannot force that $f$ is total, contradicting the hypothesis that
$r$ forced that $f$ was a Cauchy sequence.
\end{proof}
\underline{\bf Comments and Questions} Those familiar with the
proof via the (full) topological model, or sheaves, over $\mathbb
R$, as in \cite{FH} for instance, will realize that it's
essentially the same as the one above. In fact, the
topological/sheaf construction can be read off of the argument
above. All of the proofs are based on constructing the right term
and/or using an open set to force a statement. That is exactly
what's present in a topological model: the terms here are the
standard terms for a topological model, and the forcing relation
here is the standard topological semantics. So the Kripke
superstructure is actually superfluous for this argument.
Nonetheless, several questions arise.

What the Kripke structure has that the topological model doesn't
are the infinitesimals. Are they somehow hidden in the topological
model? Are they dispensable in the Kripke model? Or are the models
more than superficially different?

Also, is there some reason that the topology was necessary in the
Kripke construction? The authors started this project with the
idea of using a Kripke model, were led to infinitesimals, and did
not suspect that any topological ideas would be necessary. (In
some detail, suppose you're looking at a Dedekind cut in a node of
a Kripke model. By locatedness, if $p < q$ then one of those two
rationals gets put into either the lower or the upper cut; that
is, we can remain undecided about the placement of at most one
rational, which for simplicity we may as well take to be 0. Then
why doesn't the Cauchy sequence $1/n$ name this cut? That can
happen only if, at some later node, the cut no longer looks to be
around 0. But how can that happen if all other rationals are
already decided? Only if at this later node there are new
rationals that weren't there at the old node. This leads directly
to indexing nodes by infinitesimals, and having the cut look at
any node as though it's defining the infinitesimal at that node.
Notice that there seems to be no reason to use topology here.) It
was only after several attempts to define the terms, with their
equality and membership relations, using just the partial order
all failed that they were driven to the current, topological
solution. Since this all happened before we became aware of the
earlier Fourman-Hyland work, it is not possible that we were
somehow pre-disposed toward turning to topology. Rather, it seems
that topology is inherent in the problem. Is there some way to
make that suggestion precise and to see why it's true?

Indeed, this question becomes even more pressing in light of the
next section. There topology is used in a similar way, but the
terms and the semantics are like nothing we have seen before.
Indeed, the construction following could not be in its essence a
topological model of the kind considered so far in the literature,
since the latter always model IZF, whereas the former will falsify
Power Set (satisfying Exponentiation in its stead). So if there
were some method to read off the topology from the problem in this
section, it would be of great interest to see what that method
would give us in the next problem.

There are other, soft reasons to have included the preceding
construction, even though it adds little to the Fourman-Hyland
argument. Conceivably, somebody could want to know what the
paradigmatic Kripke model for the Cauchy and Dedekind reals
differing is, and this is it. It also provides a nice warm-up for
the more complicated work of the next section, to which we now
turn.

\section{The Dedekind Reals Are Not a Set}
\begin{theorem}
CZF$_{Exp}$ (i.e. CZF with Subset Collection replaced by
Exponentiation) does not prove that the Dedekind reals form a set.
\end{theorem}

\subsection{The Construction}
Any model showing what is claimed must have certain properties.
For one, the Dedekind reals cannot equal the Cauchy reals (since
CZF$_{Exp}$ proves that the Cauchy reals are a set). Hence the
current model takes its inspiration from the previous one. Also,
it must falsify Subset Collection (since CZF proves that the
Dedekind reals are a set). Hence guidance is also taken from
\cite{RSL1}, where such a model is built.

The idea behind the latter is that a (classical, external)
relation $R$ on $\mathbb N$ keeps on being introduced into the
model via a term $\rho$ but at a later node ``disappears"; more
accurately, the information $\rho$ contains gets erased, because
$\rho$ grows into all of $\mathbb N \times \mathbb N$, thereby
melting away into the other sets present (to give a visual image).
Since $R$ is chosen so that it doesn't help build any functions,
$\rho$ can be ignored when proving Exponentiation. On the other
hand, while you're free to include $\rho$ in an alleged full set
of relations, by the next node there is no longer any trace of
$R$, so when $R$ reappears later via a different term $\rho '$
your attempt at a full set no longer works.

In the present context, we will do something similar. The
troublesome relation will be (essentially) the Dedekind real $G$
from the previous construction. It will ``disappear" in that,
instead of continuing to change its mind about what it is at all
future nodes, it will settle down to one fixed, standard real at
all next nodes. But then some other real just like $G$ will appear
and pull the same stunt.

We now begin with the definition of the Kripke model, which
ultimately is distributed among the next several definitions.

\begin{definition} The underlying p.o. of the Kripke model
is the same as above: a (non-rooted) tree with $\omega$-many
levels, the nodes on level $n$ being the reals from $M_n$. $r'$ is
an immediate successor of $r$ iff $r$ is a real from some $M_n$,
$r'$ is a real from $M_{n+1}$, and $r$ and $r'$ are
infinitesimally close; that is, $f(r) - r'$, calculated in
$M_{n+1}$ of course, is infinitesimal, calculated in $M_n$ of
course. In other words, in $M_n$, $r$ is that standard part of
$r'$.
\end{definition}

\begin{definition} A term at a node of height $n$ is a set of the
form $\{ \langle \sigma_i, J_i \rangle \mid i \in I \} \cup \{
\langle \sigma_h, r_h \rangle \mid h \in H \}$, where each
$\sigma$ is (inductively) a term, each $J$ an open set of reals,
each $r$ a real, and $H$ and $I$ index sets, all in the sense of
$M_n$.
\end{definition}

The first part of each term is as in the previous section: at node
$r$, $J_i$ counts as true iff $r \in J_i$. The second part plays a
role only when we decide to have the term settle down and stop
changing. This settling down in described as follows.

\begin{definition} For a term $\sigma$ and real r $\in$ M$_n$,
$\sigma^{r}$ is defined inductively in M$_{n}$ on the terms as $\{
\langle \sigma_{i}^{r}, {\mathbb R} \rangle \mid \langle
\sigma_{i}, J_{i} \rangle \in \sigma \wedge r \in J_{i} \} \cup \{
\langle \sigma_h^r, {\mathbb R} \rangle \mid \langle \sigma_h, r
\rangle \in \sigma \}$.
\end{definition}

Note that $\sigma^r$ is (the image of) a set from the ground
model. It bears observation that $(\sigma^r)^s = \sigma^r$.

What determines when a term settles down in this way is the
transition function. In fact, from any node to an immediate
successor, there will be two transition functions, one the
embedding $f$ as before and the other the settling down function.
This fact of the current construction does not quite jive with the
standard definition of a Kripke model, which has no room for
alternate ways to go from one node to another. However, this move
is standard (even tame) for categorical models, which allow for
arbitrary arrows among objects. So while the standard categorical
description of a partial order is a category where the objects are
the elements of the order and there's an arrow from $p$ to $q$ iff
$p \leq q$, the category we're working in has two arrows from $p$
to $q$ (for immediate successors). If you're still uncomfortable
with this double arrow, or object to calling this object a Kripke
model, then double not the arrows but the nodes. That is, replace
each node $s$ by two nodes $s_{old}$ and $s_{new}$, and have the
two arrows go to these two separate nodes. Now you have a very
traditional Kripke model again. To save on subscripts, we will
work instead with two arrows going from $r$ to $s$.

\begin{definition} If s is an immediate successor of r, then there
are two transition functions from r to s, called f and g. f is the
elementary embedding from M$_n$ to M$_{n+1}$ as applied to terms.
g($\sigma$) = f($\sigma)^s$. Transition functions to non-immediate
successors are arbitrary compositions of the immediate transition
functions.
\end{definition}

When considering $g(\sigma$), note that $\sigma \in M_n$ and $s
\in M_{n+1}$. However, for purposes other than the transition
functions, we will have occasion to look at $\sigma^s$ for both
$\sigma$ and $s$ from $M_n$. In this case, please note that, since
$f$ is an elementary embedding, $(f(\sigma))^s = f(\sigma^s$).

It's easy to see that for $\sigma$ a (term for a) ground model
set, $f(\sigma$) is also a ground model set, and for $\tau$ from
the ground model (such as $f(\sigma$)) so is $\tau^r$. Hence in
this case $f(\sigma) = g(\sigma$).

We do not need to show that the transition functions are
well-defined, since they are defined on terms and not on
equivalence classes of terms. However, once we define =, we will
show that = is an equivalence relation and that $f$ and $g$
respect =, so that we can mod out by = and still consider $f$ and
$g$ as acting on these equivalence classes.

Speaking of defining =, we now do so, simultaneously with $\in$
and inductively on the terms, like in the previous section. In an
interplay with the settling down procedure, the definition is
different from in the previous section.

\begin{definition}
$J \Vdash \sigma =_{M} \tau$ and $J \Vdash \sigma \in_{M} \tau$
are defined inductively on $\sigma$ and $\tau$, simultaneously for
all open sets of reals J:

$J \Vdash \sigma =_{M} \tau$ iff for all $\langle \sigma_{i},
J_{i} \rangle \in \sigma \; J \cap J_{i} \Vdash \sigma_{i} \in_{M}
\tau$ and for all $r \in J \; \sigma^r = \tau^r$, and vice versa.

$J \Vdash \sigma \in_{M} \tau$ iff for all $r \in J$ there is a
$\langle \tau_{i}, J_{i} \rangle \in \tau$ and $J' \subseteq J$
such that $r \in J' \cap J_{i} \Vdash \sigma =_{M} \tau_{i}$, and
for all $r \in J \; \langle \sigma^r, {\mathbb R} \rangle \in
\tau^r$.
\end{definition}

(We will later extend this forcing relation to all formulas.)

\begin{definition}
At a node $r$, for any two terms $\sigma$ and $\tau$, $r \models
\sigma =_{M} \tau$ iff, for some $J$ with $r \in J, \; J \Vdash
\sigma =_{M} \tau$.

Also, $r \models \sigma \in_{M} \tau$ iff for some $J$ with $r \in
J, \; J \Vdash \sigma \in_{M} \tau$.
\end{definition}
Thus we have a first-order structure at each node.

\begin{corollary}
The model just defined is a Kripke model. That is, the transition
functions are $=_{M}$ and $\in_{M}$-homomorphisms.
\end{corollary}
\begin{proof}
Note that the coherence of the transition functions is not an
issue for us. That is, normally one has to show that the
composition of the transition functions from nodes $p$ to $q$ and
from $q$ to $r$ is the transition function from $p$ to $r$.
However, in our case, the transition functions were given only for
immediate successors, and arbitrary compositions are allowed. So
there's nothing about coherence to prove.

If $r \models \sigma =_M \tau$ then let $r \in J \Vdash \sigma =_M
\tau$. For $s$ an immediate successor of $r$, $s$ is
infinitesimally close to $r$, so $s \in f(J)$. Also, by
elementarity, $f(J) \Vdash f(\sigma) =_M f(\tau)$. Therefore, $s
\models f(\sigma) =_M f(\tau)$. Regarding $g$, by the definition
of forcing equality, $f(\sigma)^s = f(\tau)^s$, that is,
$g(\sigma) = g(\tau$). It is easy to see that for any term $\rho
\; {\mathbb R} \Vdash \rho =_M \rho$, so $s \in {\mathbb R} \Vdash
g(\sigma) =_M g(\tau$), and $s \models g(\sigma) =_M g(\tau$).

Similarly for $\in_M$.
\end{proof}

\begin{lemma} \label{equalitylemma2}
This Kripke model satisfies the equality axioms:
\begin{enumerate}
\item $\forall x \; x=x$

\item $\forall x, y \; x=y \rightarrow y=x$

\item $\forall x, y, z \; x=y \wedge y=z \rightarrow x=z$

\item $\forall x, y, z \; x=y \wedge x \in z \rightarrow y \in z$

\item $\forall x, y, z \; x=y \wedge z \in x \rightarrow z \in y.$

\end{enumerate}
\end{lemma}

\begin{proof}
Similar to the equality lemma from the previous section. For those
who are concerned that the new forcing relation might make a
difference and therefore want to see the details, here they come.

1: It is easy to show with a simultaneous induction that, for all
$J$ and $\sigma, J \Vdash \sigma =_{M} \sigma$, and, for all
$\langle \sigma_{i}, J_{i} \rangle \in \sigma, J \cap J_{i} \Vdash
\sigma_{i} \in_{M} \sigma$. Those parts of the definition of $=_M$
and $\in_M$ that are identical to those of the previous section
follow by the same inductive argument of the previous section. The
next clauses, in the current context, boil down to $\sigma^r =
\sigma^r$, which is trivially true, and, for $\langle \sigma_i,
J_i \rangle \in \sigma$ and $r \in J_i$, $\langle \sigma_i^r,
\mathbb R \rangle \in \tau^r$, which follows immediately from the
definition of $\tau^r$.

2: Trivial because the definition of $J \Vdash \sigma =_{M} \tau$
is itself symmetric.

3: For this and the subsequent parts, we need some lemmas.
\begin{lemma} If $J' \subseteq J \Vdash \sigma =_{M}
\tau$ then $J' \Vdash \sigma =_{M} \tau$, and similarly for
$\in_{M}$.
\end{lemma}
\begin{proof} By induction on $\sigma$ and $\tau$.
\end{proof}

\begin{lemma} If $J \Vdash \rho =_{M} \sigma$ and $J \Vdash
\sigma =_{M} \tau$ then $J \Vdash \rho =_{M} \tau$.
\end{lemma}
\begin{proof} The new part in the definition of $J \Vdash \rho =_{M}
\tau$ is that for all $r \in J \; \rho^r = \tau^r$. The
corresponding new parts of the hypotheses are that, for such $r$,
$\rho^r = \sigma^r$ and $\sigma^r = \tau^r$, from which the
desired conclusion follows immediately.

The old part of the definition follows, as before, by induction on
terms. Moreover, the proof is mostly identical. Starting with
$\langle \rho_{i}, J_{i} \rangle \in \rho$, we need to show that
$J \cap J_i \Vdash \rho_i \in \tau$. The construction of $\tau_k$
remains as above. What's new is the demand, by the additional
clause in the definition of forcing $\in$, that, for all $r \in J
\cap J_i, \; \langle \rho_i^r, \mathbb R \rangle \in \tau^r$. But
that's easy to see: $\langle \rho_i^r, \mathbb R \rangle \in
\rho^r$, by the definition of $\rho^r$, and, as we've already
seen, $\rho^r = \tau^r$.
\end{proof}
Returning to proving property 3, the hypothesis is that for some
$J$ and $K$ containing $r$, $J \Vdash \rho =_{M} \sigma$ and $K
\Vdash \sigma =_{M} \tau$. By the first lemma, $J \cap K \Vdash
\rho =_{M} \sigma, \sigma =_{M} \tau$, and so by the second, $J
\cap K \Vdash \rho =_{M} \tau$, which suffices.

4: Let $J \Vdash \rho =_{M} \sigma$ and $K \Vdash \rho \in_{M}
\tau$. We will show that $J \cap K \Vdash \sigma \in_{M} \tau$.
Let $r \in J \cap K$. By hypothesis, let $\langle \tau_{i}, J_{i}
\rangle \in \tau, J' \subseteq K$ be such that $r \in J' \cap
J_{i} \Vdash \rho =_{M} \tau_{i}$; without loss of generality $J'
\subseteq J$. By the first lemma, $J' \cap J_{i} \Vdash \rho =_{M}
\sigma$, and by the second, $J' \cap J_{i} \Vdash \sigma =_{M}
\tau_{i}$. Furthermore, $\rho^r = \sigma^r$ and $\langle \rho^r,
\mathbb R \rangle \in \tau^r$, hence $\langle \sigma^r, \mathbb R
\rangle \in \tau^r$.

5: Similar, and left to the reader.
\end{proof}

With this lemma in hand, we can now mod out of =$_M$ at each node,
and have a model in which equality is actually =.

\subsection{The Forcing Relation}
As above, we define a forcing relation $J \Vdash \phi$, with $J$
an open set of reals and $\phi$ a formula. The definition should
be read as pertaining to all formulas with parameters from a fixed
$M_n$, and is to be interpreted in said $M_n$.

\begin{definition}
For $\phi = \phi(\sigma_0, ... , \sigma_i)$ a formula with
parameters $\sigma_0, ... , \sigma_i$, $\phi^r$ is $\phi(\sigma_0^r,
... , \sigma_i^r)$.
\end{definition}

\begin{definition}
$J \Vdash \phi$ is defined inductively on $\phi$:

$J \Vdash \sigma = \tau$ iff for all $\langle \sigma_{i}, J_{i}
\rangle \in \sigma \; J \cap J_{i} \Vdash \sigma_{i} \in \tau$ and
for all $r \in J$ $\sigma^r = \tau^r$, and vice versa

$J \Vdash \sigma \in \tau$ iff for all $r \in J$ there is a
$\langle \tau_{i}, J_{i} \rangle \in \tau$ and $J' \subseteq$ J
such that r $\in J' \cap J_{i} \Vdash \sigma = \tau_{i}$ and for
all $r \in J \; \langle \sigma^r, {\mathbb R} \rangle \in \tau^r$

$J \Vdash \phi \wedge \psi$ iff $J \Vdash \phi$ and $J \Vdash
\psi$

$J \Vdash \phi \vee \psi$ iff for all $r \in J$ there is a $J'
\subseteq J$ containing $r$ such that $J \cap J' \Vdash \phi$ or
$J \cap J' \Vdash \psi$

$J \Vdash \phi \rightarrow \psi$ iff for all $J' \subseteq J$ if
$J' \Vdash \phi$ then $J' \Vdash \psi$, and, for all $r \in J$, if
${\mathbb R} \Vdash \phi^r$ then ${\mathbb R} \Vdash \psi^r$

$J \Vdash \exists x \; \phi(x)$ iff for all $r \in J$ there is a
$J'$ containing $r$ and a $\sigma$ such that $J \cap J' \Vdash
\phi(\sigma)$

$J \Vdash \forall x \; \phi(x)$ iff for all $r \in J$ and $\sigma$
there is a $J'$ containing $r$ such that $J \cap J' \Vdash
\phi(\sigma)$, and for all $r \in J$ and $\sigma$ there is a $J'$
containing $r$ such that $J' \Vdash \phi^r(\sigma)$.

\end{definition}
(Notice that in the last clause, $\sigma$ is not interpreted as
$\sigma^r$.)

\begin{lemma} \label{helpfullemma2}
\begin{enumerate}
\item For all $\phi \; \; \emptyset \Vdash \phi$.
\item If J' $\subseteq$ J $\Vdash \phi$ then J' $\Vdash \phi$.
\item If J$_{i}$ $\Vdash \phi$ for all $i$ then $\bigcup_{i} J_{i}
\Vdash \phi$.
\item J $\Vdash \phi$ iff for all r $\in$ J there is a J'
containing r such that J $\cap$ J' $\Vdash \phi$.
\item For all $\phi$, J if J $\Vdash \phi$ then for all r $\in J
\; {\mathbb R} \Vdash \phi^r$.
\item If $\phi$ contains only ground model terms, then either
${\mathbb R} \Vdash \phi$ or $\mathbb{R} \Vdash \neg \phi$.
\end{enumerate}
\end{lemma}

\begin{proof}

1. Trivial induction, as before.

2. Again, a trivial induction.

3. By induction. As in the previous section, for the case of $\rightarrow$,
you need to invoke the previous part of this lemma. All other
cases are straightforward.

4. Trivial, using 3.

5. By induction on $\phi$.

Base cases: = and $\in$: Trivial from the definitions of forcing =
and $\in$.

$\vee$ and $\wedge$: Trivial induction.

$\rightarrow$: Suppose $J \Vdash \phi \rightarrow \psi$ and $r \in
J$. We must show that ${\mathbb R} \Vdash \phi^r \rightarrow
\psi^r$. For the first clause, suppose $K \subseteq {\mathbb R}$
and $K \Vdash \phi^r$. If $K = \emptyset$ then $K \Vdash \psi^r$.
Else let $s \in K$. Inductively, since $s \in K \Vdash \phi^r$,
${\mathbb R} \Vdash (\phi^r)^s$. But $(\phi^r)^s = \phi^r$, so
${\mathbb R} \Vdash \phi^r$. Using the hypothesis on $J$,
${\mathbb R} \Vdash \psi^r$, and so by part 2 above, $K \Vdash
\psi^r$. For the second clause, let $s \in {\mathbb R}$. If
${\mathbb R} \Vdash (\phi^r)^s$ then ${\mathbb R} \Vdash \phi^r$.
By the hypothesis on $J$, ${\mathbb R} \Vdash \psi^r$, and $\psi^r
= (\psi^r)^s$.

$\exists$: If $J \Vdash \exists x \; \phi(x)$ and $r \in J$, let
$J'$ and $\sigma$ be such that $r \in J \cap J' \Vdash
\phi(\sigma)$. By induction, ${\mathbb R} \Vdash
\phi^r(\sigma^r)$. $\sigma^r$ witnesses that ${\mathbb R} \Vdash
\exists x \; \phi^r(x).$

$\forall$: Let $J \Vdash \forall x \; \phi(x)$ and $r \in J$. We
need to show that ${\mathbb R} \Vdash \forall x \; \phi^r(x)$. For
the first clause, we will show that for any $\sigma$, ${\mathbb R}
\Vdash \phi^r(\sigma)$. By part 4 above, it suffices to let $s \in
{\mathbb R}$ be arbitrary, and find a $J'$ containing $s$ such
that $J' \Vdash \phi^r(\sigma)$. By the hypothesis on $J$, for
every $\tau$ there is a $J''$ containing $r$ such that $J'' \Vdash
\phi^r(\tau)$. Introducing new notation here, let $\tau$ be
$shift_{r-s}\sigma$, which is $\sigma$ with all the intervals
shifted by $r-s$ hereditarily. So we have $r \in J'' \Vdash
\phi^r(shift_{r-s}\sigma)$. Now shift by $s-r$. Letting $J'$ be
the image of $J''$, note that $s \in J'$, the image of
$shift_{r-s}\sigma$ is just $\sigma$, and the image of $\phi^r$ is
just $\phi^r$. Since the forcing relation is unaffected by this
shift, we have $s \in J' \Vdash \phi^r(\sigma)$, as desired.

The second clause follows by the same argument. Given any $s \in
{\mathbb R}$ and $\sigma$, we need to show that there is a $J'$
containing $s$ such that $J' \Vdash (\phi^r)^s(\sigma)$. But
$(\phi^r)^s = \phi^r$, and we have already shown that for all
$\sigma \; {\mathbb R} \Vdash \phi^r(\sigma)$.

6. If $\mathbb{R} \not\Vdash \phi$, then we have to show that
$\mathbb{R} \Vdash \neg \phi$, that is $\mathbb{R} \Vdash \phi
\rightarrow \bot.$ Since $\phi^r = \phi$, the second clause in
forcing an implication is exactly the case hypothesis. The first
clause is that for all non-empty $J \subseteq \mathbb{R} \; J \not
\Vdash \phi.$ If, to the contrary, $J \Vdash \phi$ for some
non-empty $J$, then, letting $r \in J$, by the previous part of
this lemma, $\mathbb{R} \Vdash \phi^r$; since $\phi^r = \phi$,
this contradicts the case hypothesis.
\end{proof}

\begin{lemma} Truth Lemma: For any node r, r $\models \phi$
iff J $\Vdash \phi$ for some J containing r.
\end{lemma}

\begin{proof}
By induction on $\phi$, in detail.

=: Trivial, by the definition of =$_M$.

$\in$: Trivial, by the definition of $\in_M$.

$\wedge$: Trivial.

$\vee$: Trivial.

$\rightarrow$:
First we do the left-to-right direction (``if true at a node then
forced").

Suppose that for node $r$ of tree height $n$, $r \models \phi
\rightarrow \psi$. Note that for $s \in M_{n+1}$ infinitesimally
close to $r$, if ${\mathbb R} \Vdash f(\phi)^s$ then (inductively)
$f(\phi)^s$ holds at any successor node to $r$, in particular the
one labeled $s$. Since $r \models \phi \rightarrow \psi$, and at
node $s \; g(\phi) = f(\phi)^s$ and $g(\psi) = f(\psi)^s$,
$f(\psi)^s$ would also hold at $s$. Inductively $f(\psi)^s$ would
be forced by some (non-empty) open set. Choosing any $t$ from that
open set, by part 5 of this lemma, ${\mathbb R} \Vdash
(f(\psi)^s)^t$. Also $(f(\psi)^s)^t = f(\psi)^s$. So ${\mathbb R}
\Vdash f(\phi)^s$ implies ${\mathbb R} \Vdash f(\psi)^s$ for all
$s$ infinitely close to $r$. Hence by overspill the same must hold
for all $s$ in some finite interval $J$ containing $r$, and the
corresponding assertion in $M_n$: for all $s \in J$ if ${\mathbb
R} \Vdash \phi^s$ then ${\mathbb R} \Vdash \psi^s$. Note that the
same holds also for all subsets of $J$.

Suppose for a contradiction that no subset of $J$ containing $r$
forces $\phi \rightarrow \psi$. In $M_{n+1}$ let $J'$ be an
infinitesimal neighborhood around $r$. So $J' \not \Vdash f(\phi)
\rightarrow f(\psi)$. Since $J' \subseteq J$, the second clause in
the definition of $J' \Vdash f(\phi) \rightarrow f(\psi)$ is
satisfied. Hence the first clause is violated. Let $J'' \subseteq
J'$ be such that $J'' \Vdash f(\phi)$, but $J'' \not \Vdash
f(\psi)$. By part 4 of this lemma and the inductive hypothesis,
let $s \in J''$ be such that $s \not \models f(\psi)$. But $s
\models f(\phi)$. So $s \not \models f(\phi) \rightarrow f(\psi)$.
This contradicts $r \models \phi \rightarrow \psi$.

For the right-to-left direction (``if forced then true"), suppose
$r \in J \Vdash \phi \rightarrow \psi$. If $r \models \phi$, then
inductively let $r \in J' \Vdash \phi$. So $r \models \psi$, which
persists at all future nodes. Hence $r \models \phi \rightarrow
\psi$. The same argument applies unchanged to any extension of $r$
reached via a composition of only the $f$-style transition
functions. The other cases are compositions which include at least
one $g$; without loss of generality we can assume we're using $g$
to go from $r$ to an immediate extension $s$. If $s  \models
g(\phi)$, i.e. $s \models f(\phi)^s$, then by induction and by
part 5 ${\mathbb R} \Vdash f(\phi)^s$. Also, by elementarity $s
\in J \Vdash f(\phi) \rightarrow f(\psi)$. Hence, by the
definition of forcing $\rightarrow, {\mathbb R} \Vdash f(\psi)^s$,
so $s \models f(\psi)^s$, i.e. $s \models g(\psi)$.

$\exists$: If $r \models \exists x \; \phi(x)$, then let $\sigma$
be such that $r \models \phi(\sigma)$. Inductively there is a $J$
containing $r$ such that $J \Vdash \phi(\sigma)$. $J$ suffices. In
the other direction, if $r \in J \Vdash \exists x \; \phi(x)$, let
$J'$ and $\sigma$ be such that $r \in J \cap J' \Vdash
\phi(\sigma)$. Inductively $r \models \phi(\sigma)$, so $r \models
\exists x \; \phi(x)$.

$\forall$: For the left-to-right direction, suppose at node $r$
that $r \models \forall x \; \phi(x)$. If there were no interval
$J$ forcing the first clause in the $\forall$-forcing definition,
then let $J$ be an infinitesimal neighborhood around $r$. Let $s
\in J$ and $\sigma$ be such that there is no $J'$ containing $s$
such that $J' \Vdash f(\phi)(\sigma)$. Inductively $s \not \models
f(\phi)(\sigma)$, which is a contradiction.

If there were no interval $J$ forcing the second clause in the
$\forall$-forcing definition, then let $J$ be an infinitesimal
neighborhood around $r$. Let $s \in J$ and $\sigma$ be such that
there is no $J'$ containing $s$ such that $J' \Vdash
f(\phi)^s(\sigma)$. Inductively $s \not \models
f(\phi)^s(\sigma)$, i.e. $s \not \models g(\phi)(\sigma)$, which
is a contradiction.

The right-to-left direction is trivial.

\end{proof}

\subsection{The Final Proof}
It remains to show only
\begin{theorem} $M$ $\models$ CZF$_{Exp}$
\end{theorem}
and
\begin{theorem} $M$ $\models$ The Dedekind reals do not form a set.
\end{theorem}
\begin{proof}
The only axioms below, the proofs of which are essentially
different from the corresponding proofs in section 2, are Set
Induction, Strong Collection, Separation, and, of course,
Exponentiation.

\begin{itemize}
\item Infinity: As in the previous section, $\hat \omega$ will do.
(Recall that the canonical name $\hat{x}$ of any set $x \in M_n$
is defined inductively as $\{ \langle \hat{y}, \mathbb{R} \rangle
\mid y \in x \}.)$
\item Pairing: As in the previous section, given $\sigma$ and
$\tau$, $\{\langle \sigma, {\mathbb R} \rangle , \langle \tau,
{\mathbb R} \rangle \}$ will do.
\item Union: Given $\sigma$, $\{ \langle \tau, J \cap J_i \rangle
\mid$ for some $\sigma_i, \; \langle \tau, J \rangle \in \sigma_i$
and $\langle \sigma_i, J_i \rangle \in \sigma \} \cup \{ \langle
\tau, r \rangle \mid$ for some $\sigma_i, \; \langle \tau, r
\rangle \in \sigma_i$ and $\langle \sigma_i, r \rangle \in \sigma
\}$ will do.

\item Extensionality: We need to show that $\forall x \;
\forall y \; [\forall z \; (z \in x \leftrightarrow z \in y)
\rightarrow x = y]$. So let $\sigma$ and $\tau$ be any terms at a
node $r$ such that $r \models ``\forall z \; (z \in \sigma
\leftrightarrow z \in \tau)"$. We must show that $r \models
``\sigma = \tau"$. By the Truth Lemma, let $r \in J \Vdash
``\forall z \; (z \in \sigma \leftrightarrow z \in \tau)"$; i.e.
for all $r' \in J, \rho$ there is a $J'$ containing $r'$ such that
$J \cap J' \Vdash \rho \in \sigma \leftrightarrow \rho \in \tau$,
and a $J''$ containing $r'$ such that $J'' \Vdash \rho \in
\sigma^{r'} \leftrightarrow \rho \in \tau^{r'}$. We claim that $J
\Vdash ``\sigma = \tau"$, which again by the Truth Lemma suffices.

To this end, let $\langle \sigma_i, J_i \rangle$ be in $\sigma$;
we need to show that $J \cap J_i \Vdash \sigma_i \in \tau$. Let
$r'$ be an arbitrary member of $J \cap J_i$ and $\rho$ be
$\sigma_i$. By the choice of $J$, let $J'$ containing $r'$ be such
that $J \cap J' \Vdash \sigma_i \in \sigma \leftrightarrow
\sigma_i \in \tau$; in particular, $J \cap J' \Vdash \sigma_i \in
\sigma \rightarrow \sigma_i \in \tau$. It has already been
observed in \ref{equalitylemma2}, part 1, that $J \cap J' \cap J_i
\Vdash \sigma_i \in \sigma$, so $J \cap J' \cap J_i \Vdash
\sigma_i \in \tau$. By going through each $r'$ in $J \cap J_i$ and
using \ref{helpfullemma2}, part 3, we can conclude that $J \cap
J_i \Vdash \sigma_i \in \tau$, as desired. The other direction
($``\tau \subseteq \sigma"$) is analogous. Thus the first clause
in $J \Vdash ``\sigma = \tau"$ is satisfied.

The second clause is that, for each $r \in J$, $\sigma^r =
\tau^r$. That holds because $\sigma^r$ and $\tau^r$ are ground
model terms: $\sigma^r = \hat{x}$ and $\tau^r = \hat{y}$ for some
$x, y \in M_n$. If $x \not = y$, then let $z$ be in their
symmetric difference. Then no set would force $\hat{z} \in \hat{x}
\leftrightarrow \hat{z} \in \hat{y}$, contrary to the assumption
on $J$.

\item Set Induction (Schema): Suppose $r \models ``\forall x \;
((\forall y \in x \; \phi(y)) \rightarrow \phi(x))"$, where $r \in
M_n$; by the Truth Lemma, let $J$ containing $r$ force as much. We
must show $r \models ``\forall x \; \phi(x)"$; again by the Truth
Lemma, it suffices to show that $J$ forces the same.

If not, then there is a $\sigma$ and an $r \in J$ such that either
there is no $J'$ containing $r$ forcing $\phi(\sigma)$ or there is
no $J'$ containing $r$ forcing $\phi^r(\sigma)$. In $M_n$, pick
such a $\sigma$ of minimal $V$-rank. We will show that neither
option above is possible.

By the choice of $J$ and \ref{helpfullemma2} part 3, $J \Vdash
``(\forall y \in \sigma \; \phi(y)) \rightarrow \phi(\sigma)"$. If
we show that $J \Vdash ``\forall y \in \sigma \; \phi(y)"$, then
we can conclude that $J \Vdash \phi(\sigma)$, eliminating the
first option above.

Toward the first clause in forcing a universal, let $\tau$ be a
term. We claim that $J \Vdash ``\tau \in \sigma \rightarrow
\phi(\tau)"$, which suffices. Regarding the first clause in
forcing an implication, suppose $K \subseteq J$ and $K \Vdash \tau
\in \sigma$. Unraveling the definition of forcing $\in$, for each
$s \in K$, there is an $L$ containing $s$ forcing $\tau$ to equal
some term $\rho$ of lower $V$-rank than $\sigma$. By the
minimality of $\sigma$, some neighborhood of $s$ forces
$\phi(\rho)$, hence also $\phi(\tau)$ (perhaps by restricting to
$L$). By \ref{helpfullemma2} part 3, $K$ also forces $\phi(\tau)$.
Thus the first clause in $J \Vdash ``\tau \in \sigma \rightarrow
\phi(\tau)"$ is satisfied. The second clause in forcing an
implication is that, for all $r \in J$, if $\mathbb{R} \Vdash
\tau^{r} \in \sigma^{r}$ then $\mathbb{R} \Vdash
\phi^{r}(\tau^{r})$. If $\mathbb{R} \Vdash \tau^{r} \in
\sigma^{r}$, then $\tau^{r}$ is forced to be equal to some ground
model term $\hat{x}$ of lower $V$-rank than $\sigma$. By the
minimality of $\sigma$, $\mathbb{R} \Vdash \phi^{r}(\hat{x})$,
i.e. $\mathbb{R} \Vdash \phi^{r}(\tau^{r})$. Thus the first clause
in $J \Vdash ``\forall y \in \sigma \; \phi(y)"$ is satisfied.

Toward the second clause in that universal, given $\tau$ and $r
\in J$, we must find a $J'$ containing $r$ with $J' \Vdash \tau
\in \sigma^r \rightarrow \phi^{r}(\tau)$. We claim that $J$
suffices; that is, $(i)$ for all $K \subseteq J$, if $K \Vdash
\tau \in \sigma^r$ then $K \Vdash \phi^r(\tau)$, and $(ii)$ for
all $s \in J$, if $\mathbb{R} \Vdash \tau^s \in \sigma^r$ then
$\mathbb{R} \Vdash \phi^r(\tau^s)$ (using here that $(\sigma^r)^s
= \sigma^r$ and $(\phi^r)^s = \phi^r$). Regarding $(i)$, if $K
\subseteq J$ forces $\tau \in \sigma^r$, then for each $t \in K$
there is a neighborhood $L$ of $t$ forcing $\tau = \hat{x}$, for
some $x \in M_n$. If $L$ did not force $\phi^r(\hat{x})$, then, by
\ref{helpfullemma2} part 6, $\mathbb{R} \Vdash \neg
\phi^r(\hat{x})$, where $\hat{x}$ has lower rank than $\sigma$,
contradicting the choice of $\sigma$. So $L \Vdash
\phi^r(\hat{x})$, and $L \Vdash \phi^r(\tau)$. Since each $t \in
K$ has such a neighborhood, $K \Vdash \phi^r(\tau)$. Similarly for
$(ii)$: If $\mathbb{R} \Vdash \tau^s \in \sigma^r$, then $\tau^s$
has lower rank than $\sigma$; by the minimality of $\sigma$, it
cannot be the case that $\mathbb{R} \Vdash \neg \phi^r(\tau^s)$,
hence, by \ref{helpfullemma2} part 6, $\mathbb{R} \Vdash
\phi^r(\tau^s)$.

We have just finished proving that $J \Vdash ``\forall y \in
\sigma \; \phi(y)"$, and so $J \Vdash \phi(\sigma)$. Hence the
first option provided by the minimality of $\sigma$ is not
possible. The other option is that for some $r \in J$ there is no
$J'$ containing $r$ forcing $\phi^r(\sigma)$. We will show this
also is not possible.

Fix $r \in J$. By the choice of $J$ (using the second clause in
the definition of forcing $\forall$), there is a $J'$ containing
$r$ such that $J' \Vdash ``(\forall y \in \sigma \; \phi^r(y))
\rightarrow \phi^r(\sigma)"$. If we show that $J' \Vdash ``\forall
y \in \sigma \; \phi^r(y)$, then we can conclude $J' \Vdash
\phi^r(\sigma)$, for our contradiction.

Toward the first clause in forcing a universal, let $\tau$ be a
term. We claim that $J' \Vdash ``\tau \in \sigma \rightarrow
\phi^r(\tau)"$, which suffices. Regarding the first clause in
forcing an implication, suppose $K \subseteq J'$ and $K \Vdash
\tau \in \sigma$. We need to show $K \Vdash \phi^r(\tau)$. It
suffices to find a neighborhood of each $s \in K$ forcing
$\phi^r(\tau)$. So let $s \in K$. Unraveling the definition of
forcing $\in$, there is a $K'$ containing $s$ forcing $\tau$ to
equal some term $\rho$ of lower $V$-rank than $\sigma$. Shift (as
in the proof of \ref{helpfullemma2} part 5) by $r - s$. Since the
parameters in $\phi^r$ are all ground model terms, $shift_{r -
s}(\phi^r) = \phi^r$. Also, $rk(shift_{r - s}\rho) = rk(\rho) <
rk(\sigma)$. By the minimality of $\sigma$, there is some
neighborhood $L$ of $r$ forcing $\phi^r(shift_{r - s}\rho)$.
Shifting back, we get $shift_{s-r}(L)$ containing $s$ and forcing
$\phi^r(\rho)$. Then $s \in K' \cap shift_{s-r}(L) \Vdash
\phi^r(\tau)$, as desired.

The second clause in forcing an implication is that, for all $s
\in J'$, if $\mathbb{R} \Vdash \tau^s \in \sigma^s$ then
$\mathbb{R} \Vdash \phi^{r}(\tau^s)$. If $\mathbb{R} \Vdash \tau^s
\in \sigma^s$, then $\tau^s$ is forced to be equal to some ground
model term $\hat{x}$ of lower $V$-rank than $\sigma$. By the
minimality of $\sigma$, it cannot be the case that $\mathbb{R}
\Vdash \neg \phi^{r}(\hat{x})$, so, by \ref{helpfullemma2} part 6,
$\mathbb{R} \Vdash \phi^{r}(\hat{x})$, i.e. $\mathbb{R} \Vdash
\phi^{r}(\tau^s)$. Thus the first clause in $J \Vdash ``\forall y
\in \sigma \; \phi(y)"$ is satisfied.

Toward the second clause in that universal, given $\tau$ and $s
\in J'$, we must find a neighborhood of $s$ forcing $\tau \in
\sigma^s \rightarrow \phi^{r}(\tau)$. We claim that $J'$ suffices;
that is, $(i)$ for all $K \subseteq J'$, if $K \Vdash \tau \in
\sigma^s$ then $K \Vdash \phi^{r}(\tau)$, and $(ii)$ for all $t
\in J'$, if $\mathbb{R} \Vdash \tau^t \in \sigma^s$ then
$\mathbb{R} \Vdash \phi^{r}(\tau^t)$. For $(i)$, if $K \subseteq
J'$ forces $\tau \in \sigma^s$, then for each $t \in K$ there is a
neighborhood $L$ of $t$ forcing $\tau = \hat{x}$, for some $x \in
M_n$. If $L$ did not force $\phi^{r}(\hat{x})$, then, by
\ref{helpfullemma2} part 6, $\mathbb{R} \Vdash \neg
\phi^r(\hat{x})$, where $\hat{x}$ has lower rank than $\sigma$,
contradicting the choice of $\sigma$. So $L \Vdash
\phi^r(\hat{x})$, and $L \Vdash \phi^r(\tau)$. Since each $t \in
K$ has such a neighborhood, $K \Vdash \phi^r(\tau)$. Similarly for
$(ii)$: If $\mathbb{R} \Vdash \tau^t \in \sigma^s$, then $\tau^t$
has lower rank than $\sigma$; by the minimality of $\sigma$, it
cannot be the case that $\mathbb{R} \Vdash \neg \phi^r(\tau^t)$,
hence, by \ref{helpfullemma2} part 6, $\mathbb{R} \Vdash
\phi^r(\tau^t)$.

\item Exponentiation:
Let $\sigma$ and $\tau$ be terms at node $r$. Let $C$ be $\{
\langle \rho, J \rangle \mid rk(\rho) < \max(rk(\sigma), rk(\tau))
+ \omega$, and $J \Vdash \rho : \sigma \rightarrow \tau$ is a
function$\} \cup \{ \langle \hat{h}, s \rangle \mid h : \sigma^s
\rightarrow \tau^s$ is a function$\}.$ (The restriction on $\rho$
is so that $C$ is set-sized.) We claim that $C$ suffices.

Let $s$ be any immediate extension of $r$. (The case of
non-immediate extensions follows directly from this case.) If $s
\models ``\rho : f(\sigma) \rightarrow f(\tau)$ is a function",
then $s \models \rho \in f(C)$ by the first clause in the
definition of $C$. If $s \models ``\rho : g(\sigma) \rightarrow
g(\tau)$ is a function" and $\rho$ is a ground model term, then $s
\models \rho \in g(C)$ by the second clause. What we must show is
that for any node $r$ and sets $X$ and $Y$, if $r \models ``\rho :
\hat{X} \rightarrow \hat{Y}$ is a function", then some
neighborhood of $r$ forces $\rho$ equal to a ground model
function.

By the Truth Lemma, let $r \in J \Vdash ``\rho : \hat{X}
\rightarrow \hat{Y}$ is a function". We claim that for all $x \in
X$ there is a $y \in Y$ such that for each $s \in J \; s \models
\rho(\hat{x}) = \hat{y}$. If not, let $x$ be otherwise. Let $y$ be
such that $r \models \rho(\hat{x}) = \hat{y}$. For each immediate
successor $s$ of $r, \; s \models f(\rho)(f(\hat{x})) =
f(\hat{y})$. By overspill the same holds for some neighborhood
around $r$ (sans the $f$'s). If this does not hold for all $s \in
J$, let $s$ be an endpoint in $J$ of the largest interval around
$r$ for which this does hold. Repeating the same argument around
$s$, there is a $y'$ such that, for all $t$ in some neighborhood
of $s, \; t \models \rho(\hat{x}) = \hat{y'}$. This neighborhood
of $s$ must overlap that of $r$, though. So $y = y'$,
contradicting the choice of $s$. So the value $\rho(\hat{x})$ is
fixed on the whole interval $J$, and $\rho$ is forced by $J$ to
equal a particular ground model function.

\item Separation: Although CZF contains only $\Delta_0$
Separation, full Separation holds here. Let $\phi(x)$ be a formula
and $\sigma$ a term. Then $\{ \langle \sigma_i, J \cap J_i \rangle
\mid \langle \sigma_i, J_i \rangle \in \sigma$ and $J \Vdash
\phi(\sigma_i) \} \cup \{ \langle x, s \rangle \mid \langle x,
\mathbb{R} \rangle \in \sigma^s$ and ${\mathbb R} \Vdash \phi^s(x)
\}$ will do.

\item Strong Collection: If $r \models \forall x \in \sigma \;
\exists y \; \phi(x,y)$, let $r \in J$ force as much. For each
$\langle \sigma_i, J_i \rangle \in \sigma$ and $s \in J \cap J_i$,
let $\tau_{i,s}$ and $J_{i,s}$ be such that $s \in J_{i,s} \Vdash
\phi(\sigma_i, \tau_{i,s})$. Also, for each $s \in J$ and $\langle
x, \mathbb{R} \in \sigma^s$, let $\tau_{x,s}$ be such that some
neighborhood of $s$ forces $\phi^s(x, \tau_{x,s})$. (Notice that,
by \ref{helpfullemma2} part 5, $\mathbb{R} \Vdash \phi^s(x,
\tau_{x,s}^s)$.) Then $\{ \langle \tau_{i,s}, J_{i,s} \rangle \mid
i \in I, s \in J \} \cup \{ \langle \tau_{x,s}, s \rangle \mid s
\in J, x \in \sigma^s \}$ suffices.

\end{itemize}

\end{proof}
\begin{proof}
First note that the same generic term from the last section, $G :=
\{ \langle \hat{r}, J \rangle \mid  r$ is a rational, $J$ is an
open interval from the reals, and $r < J \}$, still defines a
Dedekind cut. In fact, half of the proof of such is just the
argument from the last section itself. That's because most of the
properties involved with being a Dedekind real are local. For
instance, if $s < t$ are rationals at any given node, then it must
be checked at that node whether $s \in G$ or $t \not \in G$. For
this, the earlier arguments work unchanged. The same applies to
images of $G$ at later nodes, as long as such image satisfies the
same definition, i.e. is of the form $f(G)$. We must check what
happens when $G$ settles down. Cranking through the definition,
$G^s = \{ \langle \hat{r}, \mathbb{R} \rangle \mid r < s \}$,
which is the Dedekind cut standing for $s$, and which satisfies
all the right properties.

Furthermore, although $G$ can become a ground model real, it isn't
one itself: there are no $J$ and $r$ such that $J \Vdash G =
\hat{r}$. That's because there is a $K \subseteq J$ with either $r
< s < K$ or $K < s < r$ (some $s$). In the former case, $K \Vdash
\hat{s} \in G$, i.e. $K \Vdash \hat{r} < G$; in the latter, $K
\Vdash \neg \hat{s} \in G$, i.e. $K \Vdash \hat{r} > G$.

Finally, to see that the Dedekind reals do not form a set, let
$\sigma$ be a term at any node. $g(\sigma$) is a ground model
term. So if any $J \Vdash G \in g(\sigma)$, then for some $K
\subseteq J$, $K \Vdash G = \hat{r}$ for some real $r$, which we
just saw cannot happen. So $\sigma$ cannot name the set of
Dedekind reals.
\end{proof}

Not infrequently, when some weaker axioms are shown to hold, what
interests people is not why the weaker ones are true but why the
stronger ones aren't. The failure of Subset Collection, and hence
of Power Set too, is exactly what the previous paragraph is about,
but perhaps the failure of Power Set is more clearly seen on the
simpler set 1 = $\{0\} = \{ \emptyset \}$. After applying a
settling down transition function $g$ to a set $\sigma$, the only
subsets of 1 in $g(\sigma)$ are 0 and 1. But in $M$ 1 has more
subsets than that. For instance, at node $r$, $\{ \langle
\emptyset, J \rangle \mid \min J > r\}$ is 0 at all future nodes
$s < f(r)$ and 1 at all future nodes $s > f(r)$. So $\sigma$ could
never have been the power set of 1 to begin with, because
$g(\sigma)$ is missing some such subsets.

This same example also shows that Reflection fails: that the power
set of 1 does not exist is true in $M$, as above, but not true
within any set, as once that set settles down, $\{0, 1\}$ is
indeed the internal power set of 1.

\underline{\bf Comments and Questions} The settling down process,
as explained when introduced, was motivated by the construction in
\cite{RSL1}. The change in the terms (adding members based not on
open sets but on individual real numbers, to be used only when
settling down) was quickly seen to be necessary to satisfy
Separation. But where does the unusual topological semantics come
from? The topology is the same: the space is still $\mathbb{R}$,
the only things that force statements are the same open sets as
before; it's just a change in the meaning of the forcing relation
$\Vdash$, the semantics. It is no surprise that there would have
to be some change, in the base cases (= and $\in$) if nowhere
else. But why exactly those changes as presented in the inductive
cases? The authors found them through a bothersome process of
trial and error, and have no explanation for them.

Would this new semantics have interesting applications elsewhere?
As an example of a possible kind of application, consider the
topology of this article. Under the standard semantics, there is a
Dedekind cut which is not a Cauchy sequence. With the new
semantics, the collections of Dedekind and Cauchy reals differ
(the latter being a set and the former not), but not for that
reason. In fact, any Dedekind cut is not not equal to a Cauchy
real: just apply a settling down transition. These collections
differ because the Dedekind reals include some things that are
just not yet Cauchy reals. So this semantics might be useful for
gently separating concepts, getting sets (or classes) of things to
be unequal without producing any instance of one which is not the
other.

Finally, which axioms does this new semantics verify? For
instance, in \cite{RSL3} a topological model for a generic Cauchy
sequence is given. Analogously with the current construction,
wherein a generic Dedekind cut in a model with settling down
implies that there is no set of Dedekind cuts, in a model with
settling down and a generic Cauchy sequence the Cauchy sequences
are not a set. That would mean that not even Exponentiation holds.
So what does hold generally under this semantics?

\end{document}